\documentclass[12pt, a4paper, leqno]{amsart}
\usepackage[utf8]{inputenc}
\usepackage{amsmath}
\usepackage{amsfonts}
\usepackage{amssymb}
\usepackage{times}
\usepackage[T1]{fontenc} %skandit 
\usepackage{url} 
\usepackage{color,esint}
\usepackage[colorlinks]{hyperref} 
\usepackage{mdframed}
\usepackage{pgf,tikz,pgfplots}
\pgfplotsset{compat=1.14}
\usepackage{mathrsfs}
\usetikzlibrary{arrows}

\let\oldmarginpar\marginpar
\renewcommand\marginpar[1]{\-\oldmarginpar[\raggedleft\footnotesize #1]%
{\raggedright\footnotesize #1}}

\usepackage{amsthm}
\theoremstyle{plain}
\newtheorem{thm}[equation]{Theorem}
\newtheorem{lem}[equation]{Lemma}

\newtheorem{cor}[equation]{Corollary}

\theoremstyle{definition}
\newtheorem{defn}[equation]{Definition}

\newtheorem{eg}[equation]{Example}

\theoremstyle{remark}
\newtheorem{rem}[equation]{Remark}

\numberwithin{equation}{section}

\newcommand{\R}{\mathbb{R}}
\newcommand{\N}{\mathbb{N}}
\newcommand{\Rn}{\mathbb{R}^n}

\newcommand{\mphi}[1][\phi]{\operatorname{M_{#1}}}
\newcommand{\acc}{\Gamma_{NAC}}
\newcommand{\zk}{Z_k}
\newcommand{\ok}{\widetilde{\Omega}_k}

\def\le{\leqslant}
\def\leq{\leqslant}
\def\ge{\geqslant}
\def\geq{\geqslant}
\def\rho{\varrho}
\def\vartheta{\theta}
\renewcommand{\phi}{\varphi}
\renewcommand{\epsilon}{\varepsilon}

\def\im{\operatorname{Im}}
\def\M{\operatorname{M}}
\def\mphi{\operatorname{M_{\phi}}}
\def\tmphi{\operatorname{\widetilde{M}_{\phi}}}

\def\adm{F_{adm}}
\def\rect{\Gamma_{rect}}

\date{\today}

\usepackage{url}

\usepackage{color}
\definecolor{blau}{rgb}{0.1,0.0,0.9}

\newcounter{komcounter}
\numberwithin{komcounter}{section}

\begin{document}

\title{Fuglede's theorem in generalized Orlicz--Sobolev spaces}

\author{Jonne Juusti}
\address{Department of Mathematics and Statistics,
FI-20014 University of Turku, Finland}
\email{jthjuu@utu.fi}

\begin{abstract}
In this paper, we show that Orlicz--Sobolev spaces $W^{1,\phi}(\Omega)$ can be characterized with the  ACL- and ACC-characterizations.
ACL stands for absolutely continuous on lines and ACC for absolutely continuous on curves.
Our results hold under the assumptions that $C^1(\Omega)$ functions are dense in $W^{1,\phi}(\Omega)$, and $\phi(x,\beta) \geq  1$ for some $\beta > 0$ and almost every $x \in \Omega$.
The results are new even in the special cases of  Orlicz and double phase growth.
\end{abstract}



\maketitle

%%%%%%%%%%%%%%%%%%%%%%%%%%%%%%%%%%%%%%%%%%%%%%%%%%%%%%%%

\section{Introduction}

In this paper, we study the ACL- and ACC-characterizations of Orlicz--Sobolev spaces $W^{1,\phi}(\Omega)$, where $\phi$ has generalized Orlicz growth and $\Omega \subset \Rn$ is an open set.
ACL stands for absolutely continuous on lines and ACC for absolutely continuous on curves.
Special cases of Orlicz growth include the constant exponent case $\phi(x,t) = t^p$, the Orlicz case $\phi(x,t) = \phi(t)$, the variable exponent case $\phi(x,t) = t^{p(x)}$, and the double phase case $\phi(x,t) = t^p + a(x) t^q$.
Generalized Orlicz and Orlicz--Sobolev spaces on $\Rn$ have been recently studied for example in \cite{FerHR20, CruH18, YanYY19}, and in a more general setting in \cite{AisAH16, OhnS15}.

The ACL-characterization of the classical constant exponent Sobolev spaces was given by Nikodym in \cite{Nik33}.
It states that a function $u \in L^p(\Omega)$ belongs to $W^{1,p}(\Omega)$ if and only if it has representative $\tilde{u}$ that is absolutely continuous on almost every line segment parallel to the coordinate axes and the classical partial derivatives of $\tilde{u}$ belong to $L^p(\Omega)$.
Moreover the classical partial derivatives are equal to the weak partial derivatives.
In \cite{Fug57}, Fuglede gave a finer version of this characterization, namely, the ACC-characterization.
The ACC-characterization states that a function $u \in L^p(\Omega)$ belongs to $W^{1,p}(\Omega)$ if and only if it has representative $\tilde{u}$ that is absolutely continuous on every rectifiable curve outside a family of zero $p$-modulus and the (classical) partial derivatives $\tilde{u}$ belong to $L^p(\Omega)$.

In \cite{HarHM04}, it was shown that variable exponent Sobolev space $W^{1,p(\cdot)}(\Omega)$ also has the ACL- and ACC-characterizations, if the exponent satisfies suitable conditions and $C^1(\Omega)$ functions are dense.
In Section 8 of \cite{OhnS15}, it was shown that the results hold in the space $W^{1,\phi}(\Rn)$, if $C^1(\Rn)$-functions are dense and $\phi$ satisfies certain conditions.
In this paper, we generalize the results even further.
We show that the results hold for the space $W^{1,\phi}(\Omega)$, and we do so using fewer assumptions than in \cite{HarHM04} or \cite{OhnS15}.
There are two assumptions we need to make: First that $C^1(\Omega)$ functions are dense in $W^{1,\phi}(\Omega)$. And second, that $\phi(x,\beta) \geq  1$ for some $\beta > 0$ and almost every $x \in \Omega$.
To best of our knowledge, the results are new even in the special cases of  Orlicz and double phase growth.

We base our approach on \cite{HarHM04}, but make some modifications to both make the results more general and simplify some of the results.
One difference is that we use a slightly different definition for the modulus of a curve family.
Our definition of is based on the norm, while the definition in \cite{HarHM04} is based on the modular.
The reason for defining the modulus differently has to do with the fact that modular convergence is a weaker concept than norm convergence.
Another difference with \cite{HarHM04} is that we do not use the theory of capacities.
This has two advantages:
First, the use of capacities would force us to make some extra assumptions on $\phi$.
Second, we can prove our results directly in $W^{1,\phi}(\Omega)$, for any $\Omega \subset \Rn$, whereas in \cite{HarHM04} the results are first proven in the case $\Omega = \Rn$, and this case is then used to prove the results for $\Omega \subset \Rn$.

The structure of this paper is as follows:
Section \ref{sec:preliminaries} covers preliminaries about generalized Orlicz and Orlicz--Sobolev spaces.
In section \ref{sec:modulus} we define and discuss the modulus of a curve family.
In section \ref{seq:fugLem} we prove two lemmas, which we will need in order to prove our main results.
In section \ref{sec:fugThm} we prove our main results, the ACL- and ACC-characterizations of $W^{1,\phi}(\Omega)$.

%%%%%%%%%%%%%%%%%%%%%%%%%%%%%%%%%%%%%%%%%%%%%%%%%%%%%%%%%%%%%%%%%%%

\section{Preliminaries}
\label{sec:preliminaries} 

Throughout this paper, we assume that $\Omega \subset \Rn$ is an open set.
The following definitions are as in \cite{HarH19b}, which we use as a general reference to background theory in generalized Orlicz spaces.

\begin{defn}
We say that $\phi: \Omega\times [0, \infty) \to [0, \infty]$ is a 
\textit{weak $\Phi$-function}, and write $\phi \in \Phi_w(\Omega)$, if 
the following conditions hold
\begin{itemize}
\item For every measurable $f: \Omega \to [-\infty, \infty]$ the function $x \mapsto \phi(x, |f|)$ is measurable, and for every $x \in \Omega$ the function $t \mapsto \phi(x, t)$ is non-decreasing.
\item $\displaystyle \phi(x, 0) = \lim_{t \to 0^+} \phi(x,t) =0$  and $\displaystyle \lim_{t \to \infty}\phi(x,t)=\infty$ for every $x\in \Omega$.
\item The function $t \mapsto \frac{\phi(x, t)}t$ is 
$L$-almost increasing for $t>0$ uniformly in $\Omega$. "Uniformly" means that $L$ 
is independent of $x$.
\end{itemize}
If $\phi\in\Phi_w(\Omega)$ is additionally convex and left-continuous, then $\phi$ is a 
\textit{convex $\Phi$-function}, and we write $\phi \in \Phi_c(\Omega)$.
\end{defn}

Two functions $\phi$ and $\psi$ are \textit{equivalent}, 
$\phi\simeq\psi$, if there exists $L\ge 1$ such that 
$\psi(x,\frac tL)\le \phi(x, t)\le \psi(x, Lt)$ for every $x \in \Omega$ and every $t>0$.
Equivalent $\Phi$-functions give rise to the same space with 
comparable norms.

We define the left-inverse of $\phi$ by setting
\[
\phi^{-1}(x,\tau) := \inf\{t \geq 0 : \phi(x,t) \geq \tau\}.
\]

\subsection*{Assumptions}\label{sec:assumptions}

Let us write $\phi^+_B (t) := \sup_{x \in B} \phi(x, t)$ and $\phi^-_B (t) := \inf_{x \in B} \phi(x, t)$; and abbreviate $\phi^\pm := \phi^\pm_\Omega$.
We state some assumptions for later reference. 

\begin{itemize}
\item[(A0)]
There exists $\beta \in(0,1)$ such that $\phi(x, \beta) \le 1 \le \phi(x,1/\beta)$ for almost every $x$.
\item[(A1)]\label{defn:a1}
There exists $\beta\in (0,1)$ such that,
for every ball $B$ and a.e.\ $x,y\in B \cap \Omega$,
\[
\beta \phi^{-1}(x, t) \le \phi^{-1} (y, t) 
\quad\text{when}\quad 
t \in \bigg[1, \frac{1}{|B|}\bigg].
\]
\item[(A2)]
For every $s > 0$ there exist $\beta \in (0, 1]$ and $h \in L^1(\Omega) \cap L^\infty (\Omega)$ such that
\[
\beta \phi^{-1}(x,t) \leq \phi^{-1}(y,t)
\]
for almost every $x, y \in \Omega$ and every $t \in [h(x) + h(y), s]$.
\item[(aInc)$_p$] 
There exist $L\ge 1$ such that $t \mapsto \frac{\phi(x,t)}{t^{p}} $ is $L$-almost increasing in $(0,\infty)$.
\item[(aDec)$_q$] 
There exist $L\ge 1$ such that $t \mapsto \frac{\phi(x,t)}{t^{q}} $ is $L$-almost decreasing in $(0,\infty)$.
\end{itemize}
We say that $\phi$ satisfies (aInc), if it satisfies (aInc)$_p$ for some $p>1$.
Similarly,  $\phi$ satisfies (aDec), if it satisfies (aDec)$_q$ for some $q>1$.
We write (Inc) if the ratio is increasing rather than just almost increasing, similarly for (Dec).
See \cite[Table~7.1]{HarH19b} for an interpretation of the assumptions in some special cases.

\subsection*{Generalized Orlicz spaces}\label{sec:genOrlicz}

We recall some definitions. 
We denote by $L^0(\Omega)$ the set of measurable 
functions in $\Omega$.  

\begin{defn}\label{def:Lphi}
Let $\phi \in \Phi_w(\Omega)$ and define the \textit{modular} 
$\varrho_\phi$ for $f\in L^0(\Omega)$ by 
\[
\varrho_\phi(f) 
:= \int_\Omega \phi(x, |f(x)|)\,dx.
\]
The \emph{generalized Orlicz space},
also called Musielak--Orlicz space,
is defined as the set 
\[
 L^\phi(\Omega)
 := \big\{f \in L^0(\Omega) \colon \lim_{\lambda \to 0^+} \varrho_\phi(\lambda f) = 0\big\}
\]
equipped with the (Luxemburg) norm 
\begin{equation}\label{def:norm}
\|f\|_{L^\phi(\Omega)}
:= \inf \Big\{ \lambda>0 \colon \varrho_\phi\Big(\frac{f}{\lambda}  \Big) \leq 1\Big\}.
\end{equation}
If the set is clear from the context we abbreviate $\|f\|_{L^\phi(\Omega)}$ by $\|f\|_{\phi}$.
\end{defn}

The following lemma is a direct consequence of the proof of \cite[Theorem~3.3.7]{HarH19b}.

\begin{lem}\label{lem:pointwiseCauchy}
If $(f_i)$ is a Cauchy sequence in $L^\phi(\Omega)$ such that the pointwise limit $f(x) := \lim_{i\to\infty} f_i(x)$ ($\pm\infty$ allowed) exists for almost every $x \in \Omega$, then $f$ is the limit of $(f_i)$ in $L^\phi(\Omega)$.
\end{lem}

H\"older's inequality holds in generalized Orlicz spaces with constant $2$, without restrictions on the $\Phi_w$-function
\cite[Lemma~3.2.13]{HarH19b}:
\[
\int_\Omega |f|\, |g|\, dx \le 2 \|f\|_\phi \|g\|_{\phi^*}.
\]

\begin{defn}
A function $u \in L^\phi(\Omega)$ belongs to the
\textit{Orlicz--Sobolev space $W^{1, \phi}(\Omega)$} if its weak partial derivatives $\partial_1 u, \ldots, \partial_n u$ exist and belong to the space $L^{\phi}(\Omega)$.
For $u \in W^{1,\phi}(\Omega)$, we define the norm
\[
\| u \|_{W^{1,\phi}(\Omega)} := \| u \|_{\phi} + \| \nabla u \|_{\phi}.
\]
Here $\| \nabla u \|_{\phi}$ is short for $\big{\|} | \nabla u | \big{\|}_{\phi}$.
Again, if $\Omega$ is clear from the context, we abbreviate $\| u \|_{W^{1,\phi}(\Omega)}$ by $\| u \|_{1,\phi}$.
\end{defn}

Many of our results need the assumption that $C^1(\Omega)$-functions are dense in $W^{1,\phi}(\Omega)$.
A suuficient condition is given by \cite[Theorem~6.4.7]{HarH19b}, which states that $C^\infty(\Omega)$-functions are dense in $W^{1,\phi}(\Omega)$, if $\phi$ satisfies (A0), (A1), (A2) and (aDec).
By \cite[Lemma~4.2.3]{HarH19b}, (A2) can be omitted, if $\Omega$ is bounded.

%%%%%%%%%%%%%%%%%%%%%%%%%%%%%%%%%%%%%%%%%%%%%%%%%%%%%%%%%%%%%%%%%%%

\section{Modulus of a family of curves}\label{sec:modulus}

By a curve, we mean any continuous function $\gamma: I \to \Rn$, where  $I = [a,b]$ is a closed interval.
If a curve $\gamma$ is rectifiable, we may assume that $I = [0,\ell(\gamma)]$, where $\ell(\gamma)$ denotes the length of $\gamma$. 
We denote the image of $\gamma$ by $\im(\gamma)$, and by $\rect(\Omega)$ we denote the family of all rectifiable curves $\gamma$ such that $\im(\gamma) \subset \Omega$.
Let $\Gamma \subset \rect(\Omega)$.
We say that a Borel function $u: \Omega \to [0,\infty]$ is \emph{$\Gamma$-admissible}, if
\[
\int_\gamma u \,ds
\geq 1
\]
for all $\gamma \in \Gamma$, where $ds$ denotes the integral with respect to curve length.
We denote the set of all $\Gamma$-admissible functions by $\adm(\Gamma)$.

\begin{defn}\label{def:modulus}
Let $\Gamma \subset \rect(\Omega)$.
Let $\phi \in \Phi_w(\Omega)$.
We define the \emph{$\phi$-modulus} of $\Gamma$ by
\[
\mphi(\Gamma)
:= \inf_{u\in \adm(\Gamma)} \|u\|_\phi.
\]
If $\adm(\Gamma) = \emptyset$, we set $\mphi(\Gamma) := \infty$.
A family of curves $\Gamma$ is \emph{exceptional}, if $\mphi(\Gamma) = 0$.
\end{defn}

The definition above is as in \cite{Luk16}.
The following lemma gives some useful properties of the modulus.
Items (a) and (b) are items (a) and (c) of \cite[Lemma~4.5]{Luk16}, and item (c) is a consequence of \cite[Proposition~4.8]{Luk16}.
To use the lemma, we must check that $L^{\phi}(\Omega)$ satisfies conditions (P0), (P1), (P2) and (RF) stated at the beginning of section 2 in \cite{Luk16}.
The conditions (P0), (P1) and (P2) are easy to check.
For (RF) to hold, there must exists $c \geq 1$ such that
\[
\Big{\|} \sum_{i=1}^\infty u_i \Big{\|}_\phi
\leq \sum_{i=1}^\infty c^{i}\|u_i\|_\phi
\]
holds for non-negative $u_i \in L^\phi(\Omega)$.
This is an easy consequence of \cite[Lemma~3.2.5]{HarH19b}, which states that there exists $c \geq 1$ such that
\[
\Big{\|} \sum_{i=1}^\infty u_i \Big{\|}_\phi
\leq c\sum_{i=1}^\infty \|u_i\|_\phi.
\]

\begin{lem}\label{lem:modulus}
Let $\phi \in \Phi_w(\Omega)$, then the $\phi$-modulus has the following properties:
\begin{itemize}
\item[(a)] if $\Gamma_1 \subset \Gamma_2$, then $\mphi(\Gamma_1) \leq \mphi(\Gamma_2)$,
\item[(b)] if $\mphi(\Gamma_i) = 0$ for every $i \in \N$, then $\mphi(\bigcup_{i=1}^\infty\Gamma_i) = 0$.
\item[(c)] $\mphi(\Gamma) = 0$ if and only if there exists a non-negative Borel function $u \in L^\phi(\Omega)$ such that $\int_\gamma u \,ds = \infty$ for every $\gamma \in \Gamma$.
\end{itemize}
\end{lem}

In \cite{Fug57}, the $L^p$-modulus was originally defined by
\[
\M_p(\Gamma)
:= \inf_{u \in \adm(\Gamma)} \int_\Omega u^p \,dx. 
\]
This differs from Definition \ref{def:modulus} in that the infimum is taken over the modulars of admissible functions instead of their norms.
A similar approach was taken in the variable exponent case in \cite{HarHM04}.
Following the original approach, we could have defined the modulus by
\[
\tmphi(\Gamma)
:= \inf_{u\in \adm(\Gamma)} \int_{\Omega} \phi(x, u(x)) \,dx.
\]
In the case $\phi(x,t) = t^p$, where $1\leq p<\infty$, we have $\tmphi(\Gamma) = \mphi(\Gamma)^p$.
Thus in this special case $\tmphi(\Gamma) = 0$ if and only if $\mphi(\Gamma) = 0$.
Since we are only interested in whether a family of curves is exceptional or not, in this case it does not matter whether we use $\mphi$ or $\tmphi$.

In the general case, the situation is somewhat more complicated.
Let $\phi \in \Phi_w(\Omega)$.
By \cite[Corollary~3.2.8]{HarH19b}, if $\|u\|_\phi < 1$, then $\rho_\phi(u) \lesssim \| u\|_\phi$.
Thus $\mphi(\Gamma) = 0$ implies $\tmphi(\Gamma) = 0$.
The converse implication does not necessarily hold, as the next example shows, which is the main reason for using norms instead of modulars in Definition \ref{def:modulus}.

\begin{eg}
Define $\phi \in \Phi_w(\R^2)$ by
\[
\phi(x,t) := \left\{\begin{array}{ll}
0 & \text{ if }t\leq 1, \\
t-1 & \text{ if }t>1.
\end{array}\right.
\]
For $y \in [0,1]$, let $\gamma_y:[0,1]\to\R^2,z\mapsto (y,z)$, and let $\Gamma := \{\gamma_y: y \in [0,1]\}$.
Let $u = 1$ everywhere.
Then
\[
\int_\gamma u(s) \,ds
= 1
\]
for every $\gamma \in \Gamma$, and therefore $u \in \adm(\Gamma)$.
Since $\phi(x,u(x)) = 0$ for every $x \in \R^2$, we have $\rho_\phi(u) = 0$, and thus $\tmphi(\Gamma) = 0$.

To show that $\mphi(\Gamma) > 0$, suppose on the contrary, that $\mphi(\Gamma) = 0$.
Then by Lemma \ref{lem:modulus} (c) there exists some $v \in L^\phi(\R^2)$ such that $\int_\gamma v \,ds = \infty$ for every $\gamma \in \Gamma$.
Thus
\[
\int_{[0,1]} v(y,z) \,dz
= \int_{\gamma_y} v \,ds
= \infty
\]
for every $y \in [0,1]$.
Let $\lambda > 0$.
Since $\phi(x,t) \geq t-1$ for every $x \in \R^2$ and every $t \geq 0$, Fubini's theorem implies that
\[
\int_{\R^2} \phi(x,\lambda v(x)) \,dx
\geq \int_{[0,1]}\int_{[0,1]} \lambda v(y,z) - 1 \,dz\,dy
= \infty -\int_{[0,1]}\int_{[0,1]} 1 \,dz\,dy
= \infty
\]
Since $\lambda > 0$ was arbitrary, it follows by \eqref{def:norm} that $\| v\|_\phi = \infty$.
But this is impossible, since $v \in L^\phi(\R^2)$.
Thus the assumption that $\mphi(\Gamma) = 0$ must be wrong and $\mphi(\Gamma) > 0$.
\end{eg}

Note that if $\phi \in \Phi_w(\Omega)$ satisfies (aDec)$_q$ for $1\leq q<\infty$, then, by \cite[Lemma~3.2.9]{HarH19b} (since $\phi$ satisfies (aInc)$_1$ by definition) we have
\[
\|u\|_\phi
\lesssim \max\{\rho_\phi(u),\rho_\phi(u)^{\frac{1}{q}}\}.
\]
Thus, if $\phi$ satisfies (aDec), then $\tmphi(\Gamma) = 0$ if and only if $\mphi(\Gamma) = 0$.

%%%%%%%%%%%%%%%%%%%%%%%%%%%%%%%%%%%%%%%%%%%%%%%%%%%%%%%%%%%%%%%%%%%

\section{Fuglede's Lemma}\label{seq:fugLem}

\begin{lem}[Fuglede's lemma]\label{lem:Fuglede}
Let $\phi\in\Phi_w(\Omega)$, and let $(u_i)$ be a sequence of non-negative Borel functions converging to zero in $L^\phi(\Omega)$.
Then there exists a subsequence $(u_{i_k})$ and an exceptional set $\Gamma \subset \rect(\Omega)$ such that for all $\gamma\notin\Gamma$ we have
\[
\lim_{k\to\infty} \int_\gamma u_{i_k} \,ds
= 0.
\]
\end{lem}

\begin{proof}
Let $(v_k) := (u_{i_k})$ be a subsequence of $(u_i)$, such that
\[
\| v_k\|_\phi
\leq 2^{-k}.
\]
Let $\Gamma \subset \rect(\Omega)$ be the family of curves $\gamma$, such that $\int_\gamma v_k  \,ds \nrightarrow 0$ as $k \to \infty$.
For every $j \in \N$, let
\[
w_j
:= \sum_{k=1}^j v_k.
\]
Since every $v_k$ is a non-negative Borel function, every $w_j$ is also a non-negative Borel function.
Since the sequence $(w_j(x))$ is increasing for every $x\in\Omega$, the limit $w(x) := \lim_{j\to\infty}w_j(x)$ (possibly $\infty$) exists.
By \cite[Corollary~3.2.5]{HarH19b}, if $j<m$, then
\[
\|w_m-w_j\|_\phi
=\Big{\|} \sum_{k=j+1}^m v_k \Big{\|}_\phi
\lesssim \sum_{k=j+1}^m \| v_k \|_\phi
\leq \sum_{k=j+1}^m 2^{-k}
< 2^{-j},
\]
which implies that $(w_j)$ is a Cauchy sequence in $L^\phi(\Omega)$.
By Lemma \ref{lem:pointwiseCauchy}, $w$ is the limit of $(w_j)$ in $L^\phi(\Omega)$, which implies that $w \in L^\phi(\Omega)$, and therefore $\| w \|_\phi < \infty$.

Suppose now that $\gamma \in \Gamma$.
Then
\[
\int_\gamma w \,ds
= \sum_{k=1}^\infty \int_\gamma v_k \,ds
= \infty,
\]
because $\sum_{k=1}^\infty \int_\gamma v_k \,ds < \infty$ would imply that $\lim_{k\to\infty}\int_\gamma v_k \,ds = 0$.
Thus $w/m$ is $\Gamma$-admissible for every $m\in\N$.
Since $\lim_{m\to\infty} \|w/m\|_\phi = \lim_{m\to\infty} \|w\|_\phi / m = 0$, we have $\mphi(\Gamma) = 0$.
\end{proof}

Let $E \subset \Omega$.
We denote by $\Gamma_E$ the set of all curves $\gamma \in \rect(\Omega)$, such that the $E \cap \im(\gamma)$ is nonempty.

The next lemma is, in a sense, a combination of \cite[Lemma~3.1]{HarHM04} and \cite[Lemma~5.1]{BarHH18}.
The former of the aforementioned lemmas states that if $C^1(\Rn)$ functions are dense in the variable exponent Sobolev space $W^{1,p(\cdot)}(\Rn)$ and $1 < p^{-}\leq p^+ < \infty$, then $\Gamma_E$ is exceptional whenever $E \subset \Rn$ is of capacity zero.
The latter states that if $\phi \in \Phi_w(\Rn)$ satisfies (aInc) and (aDec), then for every Cauchy sequence in $C(\Rn) \cap W^{1,\phi}(\Rn)$ there exists a subsequence which converges pointwise outside a set of zero capacity.
The beginning of the proof of our lemma is similar to \cite[Lemma~5.1]{BarHH18},  but then we use the ideas from \cite[Lemma~3.1]{HarHM04} and modify the proof to replace convergence outside a set of capacity zero by convergence outside a set $E$, such that $\Gamma_E$ is exceptional.
The reason that we do not simply prove a direct generalization of \cite[Lemma~3.1]{HarHM04} and then use \cite[Lemma~5.1]{BarHH18} is, that our proof avoids the use of capacities.
This has two advantages: First, we can drop the assumptions (aInc) and (aDec).
And second, our new result works in $W^{1,\phi}(\Omega)$ for any $\Omega \subset \Rn$, while in \cite[Lemma~3.1]{HarHM04} and \cite[Lemma~5.1]{BarHH18} we have $\Omega = \Rn$.
 
\begin{lem}\label{lem:zeroModMeas}
Let $\phi \in \Phi_w(\Omega)$ and let $(u_i)$ be a Cauchy sequence of functions in $C^1(\Omega)\cap W^{1,\phi}(\Omega)$.
Then there exists a set $E$ and a subsequence $(u_{i_k})$ such that $\mphi(\Gamma_E) = |E| = 0$ and $(u_{i_k})$ converges pointwise everywhere outside $E$.
\end{lem}

\begin{proof}
By \cite[Lemma~3.3.6]{HarH19b} there exists a subsequence of $(u_i)$ that converges pointwise almost everywhere.
Thus we can choose a subsequence $(v_k) := (u_{i_k})$, such that $(v_k)$ converges pointwise almost everywhere, and
\[
\| v_{k+1} - v_k \|_{1,\phi}
< 4^{-k}
\]
for every $k \in \N$.
For every $k \in \N$, let $f_k:= 2^k(v_{k+1} - v_k)\in C^1(\Omega) \cap W^{1,\phi}(\Omega)$.
For every $j \in \N$, let
\[
g_j
:= \sum_{k=1}^j |f_k|
\quad\text{and}\quad
h_j := \sum_{k=1}^j |\nabla f_k|.
\]
Since the sequences $(g_j(x))$ and $(h_j(x))$ are increasing for every $x \in \Omega$, the limits $g(x) := \lim_{j\to\infty} g_j(x)$ and $h(x) :=\lim_{j\to\infty} h_j(x)$ (possibly $\infty$) exist.
Since the functions $g_j$ are continuous, $g$ is a Borel function.
If $j<m$, then by \cite[Corollary~3.2.5]{HarH19b}
\[
\|g_m -g_j\|_{\phi}
\lesssim \sum_{k=j+1}^m \|f_k\|_{\phi}
\leq \sum_{k=j+1}^\infty \|f_k\|_{1,\phi}
< \sum_{k=j+1}^\infty 2^{-k}
= 2^{-j},
\]
which implies that $(g_j)$ is a Cauchy sequence in $L^{\phi}(\Omega)$.
By Lemma \ref{lem:pointwiseCauchy}, $g$ is the limit of $(g_j)$ in $L^\phi(\Omega)$.
Similarly, since
\[
\|h_m-h_j\|_\phi
\lesssim \sum_{k=j+1}^m \|\nabla f_k\|_{\phi}
\leq \sum_{k=j+1}^\infty \|f_k\|_{1,\phi}
< 2^{-j},
\]
we find that $h$ is the limit of $h_j$ in $L^\phi(\Omega)$.

Since $f_k \in C^1(\Omega)$, for any $k\in\N$ we have
\[
\big{|}|f_k(x)| - |f_k(y)|\big{|}
\leq |f_k(x) - f_k(y)|
\leq \int_\gamma |\nabla f_k| \,ds
\]
for every $x,y\in\Omega$ and any $\gamma \in \rect(\Omega)$ containing $x$ and $y$.
Thus for every $j \in \N$ we have
\begin{equation}\label{equ:upperGrad}
|g_j(x)-g_j(y)|
\leq \sum_{k=1}^j \big{|}|f_k(x)| - |f_k(y)|\big{|}
\leq \sum_{k=1}^j \int_\gamma |\nabla f_k| \,ds
= \int_\gamma h_j \,ds,
\end{equation}
for every $x,y\in\Omega$ and any $\gamma \in \rect(\Omega)$ containing $x$ and $y$.

Denote by $E$ the set of points $x \in \Omega$ such that the sequence $(v_k(x))$ does not converge.
Since $(v_k)$ converges pointwise almost everywhere, we have $|E| = 0$.
It is easy to see that if $x \in E$, then $x \in \{|f_k| > 1\}$ for infinitely many $k \in \N$, and therefore $g(x) = \infty$.
Thus
\[
E
\subset E_\infty
:= \{x \in \Omega : g(x) = \infty\},
\]
and $\Gamma_E \subset \Gamma_{E_\infty}$.
Next we construct a set $\Gamma \subset \rect(\Omega)$ such that
$\Gamma_{E_\infty} \subset \Gamma$ and $\mphi(\Gamma) = 0$.
It then follows by Lemma \ref{lem:modulus} (a) that $\mphi(\Gamma_E) = \mphi(\Gamma_{E_\infty}) = 0$.

By Lemma \ref{lem:Fuglede}, considering a subsequence if necessary, we find an exceptional set $\Gamma_1 \subset \rect(\Omega)$ such that
\[
\lim_{j\to\infty} \int_\gamma h - h_j \,ds
= 0
\]
for every $\gamma \in \rect(\Omega) \setminus \Gamma_1$.
Let
\[
\Gamma_2
:= \Big{\{} \gamma \in \rect(\Omega) : \int_\gamma g \,ds = \infty \Big{\}}
\quad \text{and} \quad
\Gamma_3
:= \Big{\{} \gamma \in \rect(\Omega) : \int_\gamma h \,ds = \infty \Big{\}}.
\]
For every $m \in \N$, the function $g/m$ is $\Gamma_2$ admissible, hence $\mphi(\Gamma_2) \leq \|g \|_\phi / m$.
Thus it follows that $\mphi(\Gamma_2) = 0$.
Similarly, we see that $\mphi(\Gamma_3) = 0$.
Let $\Gamma := \Gamma_1 \cup \Gamma_2 \cup \Gamma_3$.
By Lemma \ref{lem:modulus} (b) $\mphi(\Gamma) = 0$.

It remains to show that $\Gamma_{E_\infty} \subset \Gamma$.
Suppose that $\gamma \in \rect(\Omega)\setminus\Gamma$.
Since $\gamma \notin \Gamma_2$, there must exist some $y \in \im(\gamma)$ with $g(y) < \infty$.
By \eqref{equ:upperGrad}, for any $x \in \im(\gamma)$ and any $j \in \N$ we have
\[
g_j(x)
\leq g_j(y) + |g_j(x) - g_j(y)|
\leq g_j(y) + \int_\gamma h_j \,ds.
\]
Since $\gamma \notin \Gamma_1$, it follows that
\[
\lim_{j\to\infty} \int_\gamma h_j \,ds
= \int_\gamma h \,ds,
\]
where the right-hand side is finite because $\gamma \notin \Gamma_3$. Thus we have
\[
g(x)
= \lim_{j\to\infty} g_j(x)
\leq \lim_{j\to\infty} \Big{(}g_j(y) + \int_\gamma h_j \,ds \Big{)}
= g(y) + \int_\gamma h \,ds
< \infty.
\]
Since $x \in \im(\gamma)$ was arbitrary, it follows that $\gamma \notin \Gamma_{E_\infty}$.
And since $\gamma\notin\Gamma$ was arbitrary, it follows that $\Gamma_{E_\infty} \subset \Gamma$.
\end{proof}

%%%%%%%%%%%%%%%%%%%%%%%%%%%%%%%%%%%%%%%%%%%%%%%%%%%%%%%%%%%%%%%%%%%

\section{Fuglede's Theorem}\label{sec:fugThm}

We begin this section by defining some notations.
Let $k \in \{1,2,\dots,n\}$.
If $z \in \R$ and $y = (y_1,y_2,\dots,y_{n-1}) \in \R^{n-1}$ we define
\[
(y,z)_k
:= (y_1,\dots,y_{k-1},z,y_{k},\dots,y_{n-1}) \in \Rn.
\]
For every $x = (x_1,x_2,\dots,x_n) \in \Rn$, we write $\tilde{x}_k := (x_1,\dots,x_{k-1},x_{k+1},\dots,x_n) \in \R^{n-1}$.
With these notations, we have $x = (\tilde{x}_k,x_k)_k$.
We define $\ok \subset \R^{n-1}$ by
\[
\ok
:= \{ \tilde{x}_k : x \in \Omega \}
= \{ y \in \R^{n-1} : (y,z)_k \in \Omega \text{ for some } z \in \R \}.
\]
The set $\ok$ is, in a sense, the orthogonal projection of $\Omega$ into the space $\{x \in \Rn : x_k = 0\}$, but strictly speaking this is not true, since a projection is a function $P: \Rn \to \Rn$, but $\ok \subset \R^{n-1}$.
For every $y \in \ok$, we let $\zk(y) \subset \R$ be the set of points $z$, such that $(y,z)_k \in \Omega$.
Note that $\Omega = \{ (y,z)_k : y \in \ok\text{ and } z\in \zk(y) \}$. 

Since we will be using Lebesgue measures with different dimensions simultaneously, we will use subscripts to differentiate them, i.e. $m$-dimensional measure will be denoted by $|\cdot|_m$.

\begin{defn}
We say that $u : \Omega \to \R$ is \emph{absolutely continuous on lines}, $u \in ACL(\Omega)$, if it is absolutely continuous on almost every line segment in $\Omega$ parallel to the coordinate axes.
More formally, let $k \in \{1,2,\dots,n\}$ and let $E_k \subset \ok$ be the set of points $y$ such that the function
\[
f_y: \zk(y) \to [-\infty,\infty],\, f_y(z) = u((y,z)_k)
\]
is absolutely continuous on every compact interval $[a,b] \subset \zk(y)$.
Then $u \in ACL(\Omega)$ if and only if $|\ok \setminus E_k|_{n-1} = 0$ for every $k$.
\end{defn}

Let $u \in ACL(\Omega)$.
Absolute continuity implies that the classical partial derivative $\partial_k u$ of $u \in ACL(\Omega)$ exist for every $x \in \Omega$ such that $\tilde{x}_k \in E_k$.
Since $|\ok \setminus E_k|_{n-1} = 0$, it follows by Fubini's theorem that $\partial_k u$ exists for almost every $x \in \Omega$.
Another application of Fubini's theorem shows that the classical partial derivative is equal to the weak partial derivative, see \cite[Theorem~2.1.4]{Zie89}.
Since the partial derivatives exist almost everywhere, it follows that the gradient $\nabla u$ exists almost everywhere.
A function $u \in ACL(\Omega)$ is said to belong to $ACL^\phi(\Omega)$, if $|\nabla u| \in L^\phi(\Omega)$.

The following lemma follows immediately from the definitions of $L^\phi(\Omega)$, $ACL^\phi(\Omega)$ and $ W^{1,\phi}(\Omega)$.

\begin{lem}\label{lem:ACL}
If $\phi \in \Phi_w(\Omega)$, then $ACL^\phi(\Omega) \cap L^\phi(\Omega) \subset W^{1,\phi}(\Omega)$.
\end{lem}

\begin{defn}\label{def:ACC}
For any $u : \Omega \to R$, let $\acc(u) \subset \rect(\Omega)$ be the family of curves $\gamma : [0,\ell(\gamma)] \to \Omega$ such that $u \circ \gamma$ is not absolutely continuous on $[0,\ell(\gamma)]$.
If $\mphi(\acc(u)) = 0$, then we say that $u$ is \emph{absolutely continuous on curves}, $u \in ACC(\Omega)$.
\end{defn}

In the next lemma, we show that $ACC(\Omega)$ is a subset of $ACL(\Omega)$, if $\phi$ satisfies a suitable condition.

\begin{lem}\label{lem:ACCACL}
Let $\phi \in \Phi_w(\Omega)$ and assume that $\phi$ satisfies the following condition:
\begin{equation}\label{def:weakA0}
\text{there exist }\beta > 0 \text{ such that }\phi(x,\beta)\geq 1\text{ for almost every }x \in \Omega.
\end{equation}
Then
\[
ACC(\Omega) \subset ACL(\Omega).
\]
\end{lem}
 
\begin{rem}\label{rem:weakA0}
Note that (A0) implies \eqref{def:weakA0}, but not the other way around, since we do not assume that $\phi(x,1/\beta) \leq 1$.
We also note \eqref{def:weakA0} is equivalent to
\begin{equation}\label{def:weakA0Delta}
\text{there exist }\beta > 0 \text{ and }\delta > 0\text{ such that }\phi(x,\beta)\geq \delta\text{ for almost every }x \in \Omega.
\end{equation}
It is clear that \eqref{def:weakA0} is just a special case of \eqref{def:weakA0Delta} with $\delta = 1$.
It is also clear that \eqref{def:weakA0Delta} implies \eqref{def:weakA0}, if $\delta > 1$.
Suppose then, that $\phi$ satisfies \eqref{def:weakA0Delta} with $0 < \delta < 1$.
Then
\begin{equation}\label{equ:weakA0Frac}
\frac{\delta}{\beta}
\leq \frac{\phi(x,\beta)}{\beta}
\end{equation}
for almost every $x \in \Omega$.
By (aInc)$_1$ (which $\phi$ satisfies by definition of $\Phi_w$), there exist a constant $a \geq 1$ such that
\begin{equation}\label{equ:aIncBeta}
\frac{\phi(x,\beta)}{\beta}
\leq a\frac{\phi(x,t)}{t}
\end{equation}
for almost every $x \in \Omega$ and every $t \geq \beta$.
Choosing $t := a\beta / \delta > \beta$, it follows from \eqref{equ:weakA0Frac} and \eqref{equ:aIncBeta} that $\phi(x,a\beta/\delta) \geq 1$ for almost every $x \in \Omega$, and therefore $\phi$ satisfies \eqref{def:weakA0}.
Thus the choice $\delta = 1$ in \eqref{def:weakA0} has no special meaning, except for making notations simpler by getting rid of $\delta$.
\end{rem}

\begin{proof}[Proof of Lemma \ref{lem:ACCACL}]
Let $u \in ACC(\Omega)$, and let $k \in \{1,\dots,n\}$ and let $E_k \subset \R^{n-1}$ be as in Definition \ref{def:ACC}.
By Lemma \ref{lem:modulus}, there exists a non-negative Borel function $v \in L^\phi(\Omega)$ such that $\int_\gamma v \,ds = \infty$ for every $\gamma \in \acc(u)$.
For every $y \in \ok\setminus E_k$, let $I(y) \subset \zk(y)$ be some compact interval such that $v$ is not absolutely continuous on $I(y)$, and let $\gamma_y : [0,|I(y)|_1] \to \Omega$ be a parametrization of $I(y)$.
Since $\gamma_y \in \acc(u)$, it follows that $\int_{I(y)} v((y,z)_k) \, dz = \int_{\gamma_y} v(s) \,ds = \infty$.

From \eqref{equ:weakA0Frac} (with $\delta = 1$) and \eqref{equ:aIncBeta} we get
\[
\phi(x,t)
\geq \frac{t}{a\beta}
\]
for almost every $x \in \Omega$ and every $t \geq \beta$.
Since $\phi(x,t) \geq 0$, it follows that
\begin{equation}\label{equ:phiGegT}
\phi(x,t)
\geq \frac{t}{a\beta} - \frac{1}{a}
\end{equation}
for almost every $x \in \Omega$ and every $t \geq 0$.
Let $\lambda > \| v \|_\phi$.
By \eqref{def:norm} and Fubini's theorem we have
\begin{equation}\label{equ:1GeqInt}
\begin{aligned}
& 1
\geq \int_\Omega \phi\bigg{(}x, \frac{v(x)}{\lambda}\bigg{)} \,dx
= \int_{\ok}\int_{\zk(y)} \phi\bigg{(}(y,z)_k,\frac{v((y,z)_k)}{\lambda}\bigg{)} \,dz\,dy \\
& \geq \int_{\ok \setminus E_k} \int_{I(y)} \phi\bigg{(}(y,z)_k, \frac{v((y,z)_k)}{\lambda}\bigg{)} \,dz\,dy.
\end{aligned}
\end{equation}
By \eqref{equ:phiGegT} we have
\[
\int_{I(y)} \phi\bigg{(}(y,z)_k, \frac{v((y,z)_k)}{\lambda}\bigg{)} \,dz
\geq \int_{I(y)} \frac{v((y,z)_k)}{a\beta\lambda} \,dz - \int_{I(y)} \frac{1}{a} \, dz.
\]
Since $\int_{I(y)} v((y,z)_k) \,dz = \infty$, the first integral on the right-hand side is infinite, and since $I(y)$ is compact, the second integral is finite.
Thus
\[
\int_{I(y)} \phi\bigg{(}(y,z)_k, \frac{v((y,z)_k)}{\lambda}\bigg{)} \,dz
= \infty.
\]
Inserting this into \eqref{equ:1GeqInt}, we get 
\[
1
\geq \int_{\ok \setminus E_k} \int_{I(y)} \phi\bigg{(}(y,z)_k, \frac{v((y,z)_k)}{\lambda}\bigg{)} \,dz\,dy \\
= \int_{\ok \setminus E_k} \infty \,dy.
\]
This is possible only if $|\ok \setminus E_k|_{n-1} = 0$.
Thus $u \in ACL(\Omega)$.
\end{proof}

The next example shows that the assumption \eqref{def:weakA0} in the preceding lemma is not redundant.

\begin{eg}
Let $\Omega = \R^2$.
For $x = (y,z) \in \R^2$, let
\[
\phi(x,t) := \left\{\begin{array}{ll}
t & \text{if }y = 0, \\
0 & \text{if }y \neq 0\text{ and }t \leq |y|^{-1}, \\
t & \text{if }y \neq 0\text{ and }t > |y|^{-1}.
\end{array}\right.
\]
It easily follows from \cite[Theorem~2.5.4]{HarH19b} that $\phi \in \Phi_w(\R^2)$.
Define $u : \R^2 \to \R$ by
\[
u(y,z) := \left\{\begin{array}{ll}
0 & \text{if }y < 0, \\
1 & \text{if }y = 0, \\
2 & \text{if }y > 0.
\end{array}\right.
\]
It is trivial that $u \notin ACL(\R^2)$.
It is however the case that $u \in ACC(\R^2)$.

It is easy to see, that $\acc(u) = \Gamma_E$, where $E := \{ (y,z) \in \R^2 : y = 0\}$.
Define $v : \R^2 \to [0,\infty]$ by
\[
v(y,z) := \left\{\begin{array}{ll}
\infty & \text{if }y = 0, \\
|y|^{-1} & \text{if }y \neq 0.
\end{array}\right.
\]
Since the set
\[
\{ (y,z) \in \R^2 : v(y,z) > r \}
= \{(y,z) \in \R^2 : |y| < r^{-1} \}
\]
is open for every $r \in \R$, it follows that $v$ is a Borel function.
Fix $\gamma \in \Gamma_E$.
For every $a \in [0,\ell(\gamma)]$, we write $(y_a,z_a) := \gamma(a)$.
Now, there exists some $b \in [0,\ell(\gamma)]$ with $y_b = 0$.
Since $\gamma$ is parametrized by arc-length, we have
\[
|y_a|
= |y_a - y_b|
\leq |\gamma(a) - \gamma(b)|
\leq |a - b| 
\]
for every $a \in [0,\ell(\gamma)]$.
If $a \neq b$, then $v(\gamma(a))\geq |a-b|^{-1}$, since if $y_a = 0$, then $v(\gamma(a)) = \infty$, and if $y_a \neq 0$, then $v(\gamma(a)) = |y_a|^{-1} \geq |a-b|^{-1}$.
Thus
\[
\int_\gamma v \,ds
= \int_0^b v(\gamma(a)) \,da + \int_b^{\ell(\gamma)} v(\gamma(a)) \,da
\geq \int_0^b \frac{1}{|a-b|} \,da + \int_b^{\ell(\gamma)} \frac{1}{|a-b|} \,da
= \infty.
\]
Since this holds for all $\gamma \in \Gamma_E$, by Lemma \ref{lem:modulus} (c), to show that $\mphi(\Gamma_E) = 0$, it suffices to show that $v \in L^\phi(\R^2)$.
If $x = (y,z)$ and $y \neq 0$, then $\phi(x,v(x)) = \phi(x, |y|^{-1}) = 0$.
Thus $\phi(x,v(x)) = 0$ almost everywhere, and $\rho_\phi(v) = 0$.
By \eqref{def:norm}, it follows that $\|v\|_\phi \leq 1$, and therefore $v \in L^\phi(\R^2)$.
\end{eg}

We know that $\nabla u$ exists for every $u \in ACL(\Omega)$.
Thus, if $\phi$ satisfies \eqref{def:weakA0}, then Lemma \ref{lem:ACCACL} implies that $\nabla u$ exists for every $u \in ACC(\Omega)$.
We say that $u \in ACC^\phi(\Omega)$, if $u \in ACC(\Omega)$ and $\nabla u \in L^\phi(\Omega)$.
 
\begin{thm}[Fuglede's theorem]\label{thm:Fuglede}
Let $\phi \in \Phi_w(\Omega)$ satisfy \eqref{def:weakA0}.
If $C^1(\Omega)$-functions are dense in $W^{1,\phi}(\Omega)$, then $u \in W^{1,\phi}(\Omega)$ if and only if $u \in L^\phi(\Omega)$ and it has a representative that belongs to $ACC^\phi(\Omega)$.
In short
\[
ACC^\phi(\Omega) \cap L^\phi(\Omega)
= W^{1,\phi}(\Omega).
\]
\end{thm}

\begin{proof}
By Lemmas \ref{lem:ACL} and \ref{lem:ACCACL}, we have
\[
ACC^\phi(\Omega) \cap L^\phi(\Omega)
\subset ACL^\phi(\Omega) \cap L^\phi(\Omega)
\subset W^{1,\phi}(\Omega).
\]
Thus it suffices to show that $W^{1,\phi}(\Omega) \subset ACC^\phi(\Omega)$.
Since $|\nabla u| \in L^\phi(\Omega)$ whenever $u \in W^{1,\phi}(\Omega)$, we only have to show that $W^{1,\phi}(\Omega) \subset ACC(\Omega)$.

Suppose that $u \in W^{1,\phi}(\Omega)$.
Let $(u_i)$ be a sequence of functions in $C^1(\Omega) \cap W^{1,\phi}(\Omega)$ converging to $u$ in $W^{1,\phi}(\Omega)$.
By Lemma \ref{lem:zeroModMeas}, passing to a subsequence if necessary, we may assume that $(u_i)$ converges pointwise everywhere, except in a set $E$ with $\mphi(\Gamma_E) =|E|_n = 0$.
Let $\tilde{u}(x) := \liminf_{i\to\infty} u_i(x)$ for every $x \in \Omega$.
Since the functions $u_i$ are continuous, it follows that $\tilde{u}$ is a Borel function.
Since $u_i(x)$ converges for every $x\in\Omega\setminus E$, it follows that $\tilde{u}(x) = \lim_{i\to\infty}u_i(x)$ for $x\in\Omega\setminus E$.
By Lemma \ref{lem:pointwiseCauchy}, $u_i \to \tilde{u}$ in $L^\phi(\Omega)$, and it follows that $\tilde{u} = u$ almost everywhere.

Since $u_i \to u$ in $W^{1,\phi}(\Omega)$ we may assume, considering a subsequence if necessary, that
\[
\|\nabla u_{i+1} - \nabla u_i||_\phi < 2^{-i}
\]
for every $i \in \N$.
Since
\[
u_i = u_1 + \sum_{j=1}^{i-1} (u_{j+1} - u_j),
\]
we have $|\nabla u_i| \leq g_i$ for every $i \in \N$, where
\[
g_i = |\nabla u_1| + \sum_{j=i}^{i-1}|\nabla u_{j+1} - \nabla u_j|.
\]
Since the sequence $(g_i(x))$ is increasing for every $x\in\Omega$, the limit $g(x) := \lim_{i\to\infty} g_i(x)$ (possibly $\infty$) exists.
Since the functions $g_i$ are continuous, $g$ is a Borel function.
For every $m>n$ we have
\[
\|g_m - g_n\|_\phi
= \Big{\|}\sum_{j=n}^{m-1} |\nabla u_{j+1} - \nabla u_j|\Big{\|}_\phi 
\leq \sum_{j=n}^\infty \|\nabla u_{j+1} - \nabla u_j\|_\phi
< \sum_{j=n}^\infty 2^{-i}
< 2^{-n+1},
\]
i.e. $(g_i)$ is a Cauchy sequence in $L^\phi(\Omega)$.
Lemma \ref{lem:pointwiseCauchy} implies that $g_i \to g$ in $L^\phi(\Omega)$.

Let
\[
\Gamma_1
:= \Big{\{} \gamma \in \rect(\Omega) : \int_\gamma g \,ds = \infty \Big{\}}.
\]
Since $g/j$ is $\Gamma_1$-admissible for every $j\in\N$, we find that $\mphi(\Gamma_1) = 0$.
By Lemma \ref{lem:Fuglede}, passing to a subsequence if necessary, we find an exceptional set $\Gamma_2 \subset \rect(\Omega)$, such that
\[
\lim_{i\to\infty} \int_\gamma g - g_i \,ds
= 0
\]
for every $\gamma \in \rect(\Omega) \setminus \Gamma_2$.
The set $\Gamma_2$ has the following property: if $\gamma \in \rect(\Omega) \setminus \Gamma_2$ and $0\leq a\leq b\leq\ell(\gamma)$, then $\gamma|_{[a,b]} \in \rect(\Omega) \setminus \Gamma_2$.
The reason is that, since $g-g_i\geq 0$, we have
\[
\int_\gamma g - g_i \,ds
\geq \int_{\gamma|_{[a,b]}} g - g_i \,ds
\geq 0,
\]
and since the first term tends to zero, the middle term must also tend to zero.
Let $\Gamma := \Gamma_1 \cup \Gamma_2 \cup \Gamma_E$.
By Lemma \ref{lem:modulus} (b) $\mphi(\Gamma) = 0$.

We complete the proof by showing that $\tilde{u} \circ \gamma$ is absolutely continuous for every $\gamma \in \rect(\Omega) \setminus \Gamma$.
Let $k\in\N$ and for $j \in \{1,2,\dots,k\}$, let $(a_j,b_j) \subset [0, \ell(\gamma)]$ be disjoint intervals.
Since $\im(\gamma)$ does not intersect $E$, and $u_i \in C^1(\Omega)$ for every $i$, we have
\[
\sum_{j=1}^k |\tilde{u}(\gamma(b_j)) - \tilde{u}(\gamma(a_j))|
= \lim_{i\to\infty} \sum_{j=1}^k |u_i(\gamma(b_j)) - u_i(\gamma(a_j))|
\leq \limsup_{i\to\infty} \sum_{j=1}^k \int_{\gamma|_{[a_j,b_j]}} |\nabla u_i| \,ds.
\]
Using first the fact that $|\nabla u_i| \leq g_i$, and then the fact that $\gamma|_{[a_j,b_j]} \notin \Gamma_2$, we get
\[
\limsup_{i\to\infty} \sum_{j=1}^k \int_{\gamma|_{[a_j,b_j]}} |\nabla u_i| \,ds
\leq \limsup_{i\to\infty} \sum_{j=1}^k \int_{\gamma|_{[a_j,b_j]}} g_i \,ds
= \sum_{j=1}^k\int_{\gamma|_{[a_j,b_j]}} g \,ds.
\]
Thus
\[
\sum_{j=1}^k |\tilde{u}(\gamma(b_j)) - \tilde{u}(\gamma(a_j))|
\leq \sum_{j=1}^k\int_{\gamma|_{[a_j,b_j]}} g \,ds
\]
Since $\gamma \notin \Gamma_1$, we have $g\circ \gamma \in L^1[0,\ell(\gamma)]$, which together with the inequality above implies that $\tilde{u} \circ \gamma$ is absolutely continuous on $[0,\ell(\gamma)]$.
\end{proof}

We can combine Theorem \ref{thm:Fuglede} with Lemmas \ref{lem:ACL} and \ref{lem:ACCACL} to get the following corollary:

\begin{cor}\label{cor:Fuglede}
Let $\phi \in \Phi_w(\Omega)$ satisfy \eqref{def:weakA0}.
If $C^1(\Omega)$-functions are dense in $W^{1,\phi}(\Omega)$, then 
\[
ACC^\phi(\Omega) \cap L^\phi(\Omega)
= ACL^\phi(\Omega) \cap L^\phi(\Omega)
= W^{1,\phi}(\Omega).
\]
\end{cor}

As was noted at the end of Section \ref{sec:preliminaries}, $C^\infty(\Omega)$ functions are dense in $W^{1,\phi}(\Omega)$ if $\phi$ satisfies (A0), (A1), (A2) and (aDec).
By Remark \ref{rem:weakA0}, (A0) implies \eqref{def:weakA0}.
Thus Corollary \ref{cor:Fuglede} also holds with assumptions (A0), (A1), (A2) and (aDec), instead of \eqref{def:weakA0} and density of $C^1(\Omega)$-functions.

%%%%%%%%%%%%%%%%%%%%%%%%%%%%%%%%%%%%%%%%%%%%%%%%%%%%%%%%%%%%%%%%%%%%%%%%%%

\end{document}